\documentclass[reqno]{amsart}

\usepackage{easywncy}

\newcommand{\sh}{{\mathcyr{sh}}}

\newcommand{\h}{{\mathfrak H}}

\newcommand{\Q}{{\mathbb Q}}
\newcommand{\R}{{\mathbb R}}

\theoremstyle{plain}
\newtheorem{thm}{Theorem}[section]

\newtheorem{lem}[thm]{Lemma}
\newtheorem{prop}[thm]{Proposition}

\theoremstyle{definition}

\begin{document}


\title[Restricted sum formula and derivation relation for MZVs]{Restricted sum formula and derivation relation for multiple zeta values}

\author[T. Tanaka]{Tatsushi Tanaka}
\thanks{This work was partially supported by Kyoto Sangyo University Research Grants.}

\address{Department of Mathematics, Faculty of Science, Kyoto Sangyo University \\
Motoyama, Kamigamo, Kita-ku, Kyoto-city 603-8555, Japan}
\email{t.tanaka@cc.kyoto-su.ac.jp}



\maketitle

\begin{abstract}
Two classes of relations for multiple zeta values are handled algebraically. A restricted sum formula is proved by Eie, Liaw and Ong. The derivation relation is proved by Ihara, Kaneko and Zagier. In this paper we show the latter implies the former. 
\end{abstract}

\section{Introduction}
The multiple zeta values (abbreviated as MZVs) is a real number defined by the convergent series
\[
\zeta (k_1 ,k_2, \ldots , k_n) = \sum_{m_1>m_2> \cdots >m_n>0} \frac{1}{m_1^{k_1} m_2^{k_2} \cdots m_n^{k_n}}
\]
for positive integers $k_1, k_2, \ldots ,k_n $ with $k_1 \geq 2$. It is known that many $\Q$-linear relations hold among MZVs. Moreover, MZVs have the so-called double shuffle product rule, which deepens the theory of MZVs. 

Let $a,\,b\geq 0,\ k\geq a+b+2$. As a generalization of the sum formula \cite{Gra,Zag}, the restricted sum formula
\begin{equation}\label{eqn1}
\sum_{\begin{subarray}{c} k_1+\cdots +k_b=k-a \\ k_1\geq 2, k_2,\ldots ,k_b\geq 1 \end{subarray}} \zeta(k_1, \ldots , k_b,\underbrace{1, \ldots ,1}_{a}) = \sum_{\begin{subarray}{c} k^{\prime}_1+\cdots +k^{\prime}_{a+1}=a+b+1 \\ k^{\prime}_1,\ldots ,k^{\prime}_{a+1}\geq 1 \end{subarray}} \zeta(k^{\prime}_1+k-a-b-1, k^{\prime}_2, \ldots , k^{\prime}_{a+1})
\end{equation}
for MZVs is proved in \cite{ELO}
. 
The derivation relation, which is described in the next section, is proved in \cite{IKZ} by reducing it to the regularized double shuffle relation. In \cite{Tan}, the derivation relation (and its certain extension) is proved as a part of Kawashima relation \cite{Kaw}. 

In this paper we show the derivation relation implies the restricted sum formula. Detailed statement is given in the next section.

\section{Main theorem}\label{sec2}
First the restricted sum formula \eqref{eqn1} and the derivation relation for MZVs are given in terms of the algebraic setup introduced in \cite{Hof}. Let $\h =\Q\langle x,y \rangle $ denote the noncommutative polynomial algebra over the rational numbers in two indeterminates $x$ and $y$, and let $\h^1$ and $\h^0$ denote the subalgebras $\Q +\h y$ and $\Q+x \h y$, respectively. The $\Q$-linear map $\mathit{Z}:\h^0 \to \R$ is defined by $\mathit{Z}(1)=1$ and 
\[
\mathit{Z}(x^{k_1-1}yx^{k_2-1}y\cdots x^{k_n-1}y)=\zeta(k_1,k_2,\ldots ,k_n).
\]

The shuffle product $\sh\colon \h \times \h \to \h$ is the $\Q$-bilinear map defined by 
\[
1\ \sh\ w =w\ \sh\ 1 =w
\]
for $w\in \h$ and the recursive rule
\[
uw\ \sh\ vw^{\prime} = u(w\ \sh\ vw^{\prime}) +v(uw\ \sh\ w^{\prime})
\]
for $u,\,v\in\{x,\,y\}$ and $w,\,w^{\prime}\in \h$. It is well known that the algebra $(\h,\sh)$ is commutative and associative, and $(\h^1,\sh)$ and $(\h^0,\sh)$ are its subalgebras. The map $Z$ is a $\sh$-homomorphism, {\em i.e.}
\[
Z(w\ \sh\ w^{\prime})=Z(w)Z(w^{\prime})
\]
for $w,\,w^{\prime}\in\h^0$. This is called the shuffle product formula for MZVs. 

The restricted sum formula \eqref{eqn1} is then stated, for $a,\,b\geq 0,\ k\geq a+b+2$, as
\begin{equation}\label{eqn2}
x(x^{k-a-b-2}\ \sh\ y^b)y^{a+1}-x^{k-a-b-1}(x^b\ \sh\ y^a)y \in \ker Z.
\end{equation}

A derivation $\partial$ on $\h$ is a $\Q$-linear endomorphism of $\h$ satisfying the Leibniz rule
\[
\partial(ww^{\prime})=\partial(w)w^{\prime}+w\partial(w^{\prime}). 
\]
Such a derivation is uniquely determined by its images of generators $x$ and $y$. Let $z=x+y$. For each $n\ge 1$, the derivation $\partial_n:\h \to \h$ is defined by $\partial_n(x)=xz^{n-1}y$ and $\partial_n(y)=-xz^{n-1}y$. It follows immediately that $\partial_n(\h)\subset\h^0$. The derivation relation for MZVs is then stated, for any $n\geq 1$, as
\begin{equation}\label{eqn3}
\partial_n(x \h y)\subset \ker Z.
\end{equation}

Here we state our main theorem. 
\begin{thm}[main theorem]\label{main}
The restricted sum formula \eqref{eqn2} is written as a linear combination of the derivation relation \eqref{eqn3}. That is, for $a,\,b\geq 0,\ k\geq a+b+2$, we have
\[
x(x^{k-a-b-2}\ \sh\ y^b)y^{a+1}-x^{k-a-b-1}(x^b\ \sh\ y^a)y\ \in\ \sum_{n\geq 1}\partial_n(x \h y).
\]
\end{thm}
\noindent It is known that the derivation relation is contained in Ohno's relation \cite{Ohn}. (See \cite{IKZ}.) Hence the theorem asserts that the restricted sum formula is also contained in Ohno's relation. 

\section{Derivation relation and Kawashima relation}\label{sec3}
Let $z_k=x^{k-1}y\ (\in\h y)$ for $k\geq 1$. The harmonic product $\ast\colon \h^1 \times \h^1 \to \h^1$ is the $\Q$-bilinear map defined by 
\[
1 \ast w =w \ast 1 =w
\]
for $w\in \h^1$ and the recursive rule
\[
z_kw \ast z_lw^{\prime} = z_k(w \ast z_lw^{\prime}) +z_l(z_kw \ast w^{\prime})+z_{k+l}(w \ast w^{\prime})
\]
for $k,\,l\geq 1$ and $w,\,w^{\prime}\in \h^1$. It is well known that the algebra $(\h^1,\ast)$ is commutative and associative, and $(\h^0,\ast)$ is its subalgebra. The map $Z$ is a $\ast$-homomorphism, {\em i.e.}
\[
Z(w \ast w^{\prime})=Z(w)Z(w^{\prime})
\]
for $w,\,w^{\prime}\in\h^0$. This is called the harmonic product formula for MZVs. 

Let $L_x$ denote the $\Q$-linear map on $\h$ defined for $w\in\h$ by $L_x(w)=xw$. Let $\varphi$ denote the automorphism on $\h$ characterized by $\varphi (x)=x+y$ and $\varphi (y)=-y$. We see the automorphism $\varphi$ is an involution. The derivation relation \eqref{eqn3} is rewritten by using the formulation of Kawashima relation as 
\begin{equation}\label{eqn4}
\exp\hspace{-2pt}\left(\sum_{n\geq 1}\frac{\partial_n}{n}\right)(x \h y)=L_x\varphi\left(\frac{1}{1+y} \ast \h y\right). 
\end{equation}
(See \cite[Appendix A]{Tan} for details.)
\begin{prop}\label{prop1}
For $a,\,b,\,c\geq 0$, we have
\begin{equation}\label{eqn5}
(z^c\ \sh\ (-y)^b)(-y)^{a+1}-z^c(z^b\ \sh\ (-y)^a)(-y)\ \in\ \sum_{n\geq 1}\left(y^n \ast \h y\right).
\end{equation}
\end{prop}
\noindent Applying $L_x\varphi$ to \eqref{eqn5} (for $c=k-a-b-2$) and using \eqref{eqn4}, we obtain our main theorem.

\section{Key identity}\label{sec4}
For $a\geq 0$, we define the sequence $\{F_n(a)\}_{n\geq 0}$ by 
\[
\sum_{j=0}^n (-y)^j \ast F_{n-j}(a)=(-y)^{a+1+n}
\]
for $n\geq 0$. We find that $F_n(a)\in\h y$ for any $n,\,a\geq 0$. Multiplying $X^n$ and taking the sum $\sum_{n\geq 0}$, we notice that
\begin{equation}\label{eqn8}
\frac{1}{1+yX} \ast \sum_{j\geq 0} F_j(a)X^j=\frac{(-y)^{a+1}}{1+yX}. 
\end{equation}

Then we have the following key identity. 
\begin{prop}\label{prop2}
For $a,\,b,\,c\geq 0$, we have
\begin{equation}\label{eqn6}
(z^c\ \sh\ (-y)^b)(-y)^{a+1}-z^c(z^b\ \sh\ (-y)^a)(-y)=-\sum_{j=0}^{b-1}(-y)^{b-j} \ast z^c F_j(a).
\end{equation}
\end{prop}
\noindent We easily see that Proposition \ref{prop2} implies Proposition \ref{prop1}. 

We define $\Phi_{-yY}\colon \h^1 \to \h^1$ by 
\[
\Phi_{-yY}(w)=(1+yY)\left(\frac{1}{1+yY} \ast w\right). 
\]
Then $\Phi_{-yY}$ is an automorphism of $\h^1$ and extends to an automorphism of $\h$, with $\Phi_{-yY}(x)=x$ and $\Phi_{-yY}(z)=z\frac{1}{1+yY}$. (See \cite[Proposition 6]{IKZ} for details.)

We also define the map $d\colon \h \to \h$ by $d(w)=y\ \sh\ w-yw$. Then $d$ is a derivation on $\h$ and we have
\[
\exp (du)(w)=(1-yu)\left(\frac{1}{1-yu}\ \sh\ w\right),
\]
where $u$ is a formal parameter. We find that $\exp (du)$ is an automorphism on $\h$ characterized by $\exp (du)(x)=x\frac{1}{1-yu}$ and $\exp (du)(y)=y\frac{1}{1-yu}$. We also see $\exp (du)^{-1}=\exp (-du)$. (See \cite[Proposition 7]{IKZ} for details.)

To prove Proposition \ref{prop2}, we use the following two lemmas. The proofs are given in the next section. 
\begin{lem}\label{lem1}
The identity
\[
\frac{1}{1+yY}\,\Phi_{-yY}\hspace{-2pt}\left(\frac{1}{1-zZ}\right)=\frac{1}{1-zZ}\ \sh\ \frac{1}{1+yY}
\]
holds. 
\end{lem}
\begin{lem}\label{lem2}
The identity
\[
\Phi_{-yY}^{-1}\hspace{-2pt}\left(\frac{-y}{1+yX}\right)=\left(\frac{1}{1-zY}\ \sh\ \frac{1}{1+yX}\right)(-y)
\]
holds. 
\end{lem}
Assuming these lemmas, we prove Proposition \ref{prop2}. 
\begin{proof}[Proof of Proposition \ref{prop2}]
Considering generating functions is convenient for the proof. Using Lemma \ref{lem1} and \ref{lem2}, we see that
\begin{align*}
\sum_{a,b,c\geq 0} (\text{LHS of \eqref{eqn6}}) X^aY^bZ^c&=\left(\frac{1}{1-zZ}\ \sh\ \frac{1}{1+yY}\right)\frac{-y}{1+yX}-\frac{1}{1-zZ}\left(\frac{1}{1-zY}\ \sh\ \frac{1}{1+yX}\right)(-y) \\
&=\frac{1}{1+yY}\,\Phi_{-yY}\hspace{-2pt}\left(\frac{1}{1-zZ}\right)\frac{-y}{1+yX}-\frac{1}{1-zZ}\,\Phi_{-yY}^{-1}\hspace{-2pt}\left(\frac{-y}{1+yX}\right).
\end{align*}
On the other hand, 
\begin{align*}
\sum_{a,b,c\geq 0} (\text{RHS of \eqref{eqn6}}) X^aY^bZ^c&=-\sum_{a,c\geq 0}\sum_{i\geq 1,j\geq 0}(-y)^i \ast z^cF_j(a)X^aY^{i+j}Z^c \\
&=\left(\frac{1}{1+yY}-1\right) \ast \frac{1}{1-zZ}\sum_{a,j\geq 0}F_j(a)X^aY^j \\
&=\frac{1}{1+yY} \ast \frac{1}{1-zZ}\,\Phi_{-yY}^{-1}\hspace{-2pt}\left(\frac{-y}{1+yX}\right)-\frac{1}{1-zZ}\,\Phi_{-yY}^{-1}\hspace{-2pt}\left(\frac{-y}{1+yX}\right). 
\end{align*}
The last equality is because of \eqref{eqn8} or $\sum_{j\geq 0}F_j(a)Y^j=\Phi_{-yY}^{-1}\hspace{-2pt}\left((-y)^{a+1}\right)$. By definition of $\Phi_{-yY}$, we have
\begin{align*}
\frac{1}{1+yY} \ast \frac{1}{1-zZ}\,\Phi_{-yY}^{-1}\hspace{-2pt}\left(\frac{-y}{1+yX}\right)&=\frac{1}{1+yY}\,\Phi_{-yY}\hspace{-2pt}\left(\frac{1}{1-zZ}\,\Phi_{-yY}^{-1}\hspace{-2pt}\left(\frac{-y}{1+yX}\right)\right) \\
&=\frac{1}{1+yY}\,\Phi_{-yY}\hspace{-2pt}\left(\frac{1}{1-zZ}\right)\frac{-y}{1+yX}.
\end{align*}
Therefore we conclude the proposition. 
\end{proof}

\section{Proofs of lemmas}\label{sec5}
We prove two lemmas in the previous section. 
\begin{proof}[Proof of Lemma \ref{lem1}]
By \cite[Corollary 3]{IKZ}, we have
\[
\Phi_{-yY}=\exp (-dY)\circ\Delta_{-Y}. 
\]
Here $\Delta_{-y}$ is the automorphism on $\h$ characterized by $\Delta_{-y}(x)=x\frac{1}{1-yY}$ and $\Delta_{-y}(z)=z$. Hence
\begin{align*}
\frac{1}{1+yY}\Phi_{-yY}\hspace{-2pt}\left(\frac{1}{1-zZ}\right)&=\frac{1}{1+yY}\,\exp (-dY)\hspace{-2pt}\left(\Delta_{-Y}\left(\frac{1}{1-zZ}\right)\right) \\
&=\frac{1}{1+yY}\,\exp (-dY)\hspace{-2pt}\left(\frac{1}{1-zZ}\right) \\
&=\frac{1}{1-zZ}\ \sh\ \frac{1}{1+yY}
\end{align*}
to obtain the lemma. 
\end{proof}
\begin{proof}[Proof of Lemma \ref{lem2}]
Since
\[
\frac{1}{1-zY}\ \sh\ \frac{1}{1+yX}=\frac{1}{1+yX}\,\exp (-dY)\hspace{-2pt}\left(\frac{1}{1-zY}\right)=\frac{1}{1+yX-zY}, 
\]
we have to show is
\begin{equation}\label{eqn7}
\Phi_{-yY}\hspace{-2pt}\left(\frac{1}{1+yX-zY}(-y)\right)=\frac{-y}{1+yX}.
\end{equation}
For the left-hand side, we have
\begin{align*}
\Phi_{-yY}\hspace{-2pt}\left(\frac{1}{1+yX-zY}(-y)\right)&=\frac{1}{1+\Phi_{-yY}(y)X-\Phi_{-yY}(z)Y}\left(-\Phi_{-yY}(y)\right) \\
&=-\frac{1}{1+\left(y-z\frac{yY}{1+yY}\right)X-z\frac{1}{1+yY}Y}\left(y-z\frac{yY}{1+yY}\right).
\end{align*}
We find
\begin{align*}
\frac{1}{1+\left(y-z\frac{yY}{1+yY}\right)X-z\frac{1}{1+yY}Y}&=(1+yY)\frac{1}{1+yY+\left(y(1+yY)-zyY\right)X-zY} \\
&=(1+yY)\frac{1}{(1-xY)(1+yX)}
\end{align*}
and
\[
y-z\frac{yY}{1+yY}=\left(y(1+yY)-zyY\right)\frac{1}{1+yY}=(1-xY)\frac{y}{1+yY},
\]
and hence conclude \eqref{eqn7} and the lemma. 
\end{proof}

\section*{Acknowledgements}
The author expresses his gratitude to Yasuo Ohno for valuable comments.


\end{document}